\documentclass[oneside]{article}

\usepackage{latexsym,amsmath,amssymb,amsthm}

\newtheorem{theorem}{Theorem}[section]
\newtheorem{corollary}{Corollary}[section]
\newtheorem{lemma}{Lemma}[section]

\begin{document}

\numberwithin{equation}{section}

\title{The $2$-adic Analysis of Stirling Numbers of the Second
Kind via Higher Order Bernoulli Numbers and Polynomials}

\author{Arnold Adelberg\\
Myra Steele Emeritus Professor of Mathematics\\
Department of Mathematics and Statistics\\
Grinnell College\\
Grinnell, IA 50112}

\maketitle

\begin{abstract}
Several new estimates for the $2$-adic valuations of Stirling numbers of the 
second kind are proved. These estimates, together with criteria for when they
are sharp, lead to improvements in several known theorems and their
proofs, as well as to new theorems. The estimates and criteria all depend on
our previous analysis of powers of $2$ in the denominators of coefficients of
higher order Bernoulli polynomials. The corresponding estimates for Stirling
numbers of the first kind are also proved.

Some attention is given to asymptotic cases, which will be further explored in
subsequent publications.

\end{abstract}

\section{Introduction}

This paper brings together and extends a collection of related results on the $2$-adic
analysis of Stirling numbers of the second kind. We hope that our approach, based on
our earlier results for higher order Bernoulli numbers and polynomials, provides a
coherent theoretical basis that others will find useful for further investigations.
The proofs we give for known results are shorter and simpler, often dramatically so.
The results themselves are typically sharper and broader. We also get some new
results, most of which involve new estimates that are stronger than those in the
literature.

The current paper is a continuation of [4] but is quite different in its goals and scope. 
Whereas the previous paper considered all primes and Stirling numbers of both kinds,
for reasons of brevity and focus this paper will primarily consider only the even
prime and will concentrate on Stirling numbers of the second kind $S(n,k)$.

Lengyel [11] proved in 1994 that $\nu_2(S(2^h,k)) = \sigma_2(k) -1$, if $h$ is sufficiently 
large and $k>0$, and conjectured that this formula holds whenever $1 \le k \le 2^h$,
where $\sigma_2(k) =$ sum of base $2$ digits. This was eventually proven in 2005 by
De Wannemaker [7]. Subsequently Lengyel [13] gave another proof and 
adapted De Wannemaker's 
proof to extend the theorem to $\nu_2(S(c2^h,k))=\sigma_2(k)-1$ if $c \ge 1$ and
$1 \le k \le 2^h$.

We found a much simpler proof of De Wannemaker's Theorem in [4], which we were
able to generalize to odd primes and to minimum zero cases (MZC), which are based
on the estimate 
\begin{align}
\nu_2(S(n,k)) \ge \sigma_2(k) -\sigma_2(n)
\end{align}
which we call the minimum zero estimate. When it is sharp, we have the minimum
zero case (MZC).

In the current paper, we give several other useful estimates. One, which is based on
recursive properties of Stirling polynomials, is
\begin{align}
\nu_2(S(n,k)) \ge \sigma_2(k-1)-\sigma_2(n-1)
\end{align}
which we call the shifted minimum zero estimate. When this is sharp, we have the
shifted minimum zero case (SMZC).

Significantly better than these estimates are our new almost minimum zero estimate
\begin{align}
\nu_2(S(n,k)) \ge \sigma_2(k) -\sigma_2(n) + \#(\text{common 2-powers in~} n
 \text{~and~} n-k)
\end{align}
and our new shifted almost minimum zero estimate
\begin{align}
\nu_2(S(n,k)) \ge \sigma_2(k-1)-\sigma_2(n-1) +\#(\text{common 2-powers in~} n-1
 \text{~and~} n-k)
\end{align}

An almost minimum zero case (AMZC) is one where the estimate (1.3) is sharp, but which
is not a MZC, while a shifted almost minimum zero case (SAMZC) is one which is not
a SMZC and the estimate (1.4) is sharp. When the distinction between MZC and AMZC
is unimportant, we may use AMZC for the sharp almost minimum zero estimate. We
may also adopt the analogous convention for SAMZC.

Unlike the minimum zero and shifted minimum zero estimates, these estimates are never
vacuous (negative). This leads to very simple new necessary and sufficient conditions for
when a Stirling number $S(n,k)$ is odd. [Theorem 3.2]

Most of the significant analysis of this paper rests on the fact that since $B_{n-k}^{(-k)}$
is a cofficient of $B_{n-k}^{(-k)}(x)$ the $2$-adic pole of $B_{n-k}^{(-k)}$, i.e., the 
highest power of $2$ in its denominator, is less than or equal to the maximum pole of
$B_{n-k}^{(-k)}(x)$, which is the highest power of $2$ in the denominators of
all the coefficients. We have a
simple formula for this maximum pole (cf. [1,2]), which is given in the Appendix.

The geometry of these cases is instructive: The MZC occurs when the Newton polygon
of $B_{n-k}^{(-k)}(x)$ is strictly decreasing; the SMZC occurs when the Newton
polygon of $B_{n-k}^{(-k+1)}(x)$  is strictly decreasing; the AMZC occurs when the 
Newton polygon of $B_{n-k}^{(-k)}(x)$ is weakly decreasing, i.e., the last segment
of the Newton polygon is horizontal; the SAMZC  holds when the pole of 
$B_{n-k}^{(-k+1)}(1)$ is the maximum pole of $B_{n-k}^{(-k+1)}(x)$ , but this pole 
also occurs in at least one coefficient other than the constant coefficient.  

In our study of the literature, we have found that every significant estimate or exact value 
of $\nu_2(S(n,k))$ we considered arises from one of our estimates or cases.
For example, the proofs in ([9], Theorems 1.1, 1.2, 1.3) are very lengthy and
highly technical, while ours are much shorter and more efficient.

Also, in [12] Lengyel gives many proofs of estimates for $\nu_2(S(c2^h,2^h+a))$,
which we handle easily by our methods with far less computation. He also gives
estimates for $\nu_2(S(c2^h+u,k))$ which are not as good as our almost
minimum zero estimates (unless $u$ is a power of $2$), and his proofs are more
involved than ours.

The organization of this paper is as follows: Section 2 states a number of elementary,
useful facts about base two arithmetic, gives some basic definitions, and states the
main theorems of this paper. Our statements include the estimate or case that leads to
our proof, since that provides insight into why the theorems are true.
Section 3 provides simple, effective criteria for the four cases,
establishes certain invariance properties for these estimates and cases,
and proves a couple of new theorems. Included in this section are new necessary and
sufficient conditions for the Stirling number $S(n,k)$ to be odd, which generalize our
conditions for the central Stirling numbers $S(2k,k)$ (cf. [4]). We also state,
for reference, the
estimates and cases for Stirling numbers of the first kind. Section 4 proves the
main non-asymptotic theorems. Section 5 proves an illustrative asymptotic theorem,
which is more simply stated and with an exponentially better estimate for when
the limit is attained than in the literature. Our proof in this section does not depend
on ths estimates given in the Introduction but depends instead on a new estimate for the
partition dependent terms, which is given in the Appendix [Theorem 6.1]. 
Section 6 collects the material
on higher order Bernoulli numbers and polynomials needed for this paper.

\section{Base two preliminaries, definitions, and statement of main theorems}

Since we deal only with the prime two in this paper, we will omit the prime in our 
notations; e.g., we will write $\nu$ instead of $\nu_2$ for the $2$-adic valuation and
$\sigma$ instead of $\sigma_2$ for the sum of the base two digits, which is the
same as the number of powers of two in the base two representation.

We extend our previous notion of pole to allow a $2$-adic unit to be considered as a pole
of order zero, so if $N \ge 0$ and $\nu(a) \ge -N$, then we say that $a$ has at most a
pole of order $N$.

Let $[n]=$ set of $2$-powers in its base two expansion, so if $n= \sum a_i2^i$ with
all $a_i \in \{0,1\}$, then $[n]$ corresponds to the ones in the expansion, i.e., the ones
in the base two representation of $n$. The following facts are obvious but useful:
\begin{align}
&\#([n])=\sigma(n), \quad \text{min}([n])=2^{\nu(n)}, \quad \text{max}([n])=
 2^{\lfloor \log_2(n)\rfloor}, \notag\\
&[n-1]=\left([n]-\{2^{\nu(n)}\}\right) \cup \{1,2,\ldots, 2^{\nu(n)-1}\}~ 
 (\text{disjoint union}),
 \notag\\
\text{so~}&\sigma(n-1)=\sigma(n)-1+\nu(n) ~(cf. [4]), \text{~and}\notag\\
&2^{\nu(n-m)} = \text{min}([n] \cup [m]-[n] \cap [m]) 
\end{align}
The basic facts that we need about binomial coefficients are as follows:
\begin{align}
\nu\binom{a+b}{a} = \sigma(a)+\sigma(b)-\sigma(a+b) = \#(\text{base two carries for~}
 a+b)
\end{align}

If $a= \sum a_i2^i$ and $b=\sum b_i 2^i$, then we have a base $2$ carry if either
$a_i=b_i=1$, or if $\{a_i,b_i\} = \{0,1\}$ and the carry results from previous carries. A carry 
where $a_i=b_i=1$ is called a forced carry, and other carries are said to be unforced. The
number of forced carries is $\#([a] \cap [b])$.

\begin{lemma} If $n \ge m$ then $\#([n] \cap [n-m]) \ge \sigma(n)-\sigma(m)$, i.e., 
$\#([n]-[n-m]) \le \sigma(m)$, using the set difference.
\end{lemma}
\begin{proof}
If $m=2^i$, then $n-m$ removes the smallest $2$-power in $[n]$ which is greater
than or equal to $2^i$, and if this power is bigger than $2^i$ we insert the powers down to 
$2^i$. We continue subtracting $m$ by subtracting its $2$-powers one at a time,
iterating the process.
\end{proof}

\begin{lemma} If $\binom{b}{a}$ is odd, then $\#([b] \cap [b-a]) = \#([b-a])=
\sigma(b)-\sigma(a)$; if $b$ is odd and greater than $1$, then 
$\#([b] \cap [b-3])=\sigma(b)-2$, while
if $b$ is even and greater than $2$, then $\#[b] \cap [b-3]) = \sigma(b)-1$.
\end{lemma}
\begin{proof}
The first assertion follows from (2.2). The other parts follow from $2^\alpha -2
= 2^{\alpha -1} + \cdots + 2$, first subtracting $2$ from $b$, then subtracting $1$.
\end{proof}

Next we list the main theorems, most of which come from [9,12]. We have edited them
them to conform to our notations and conventions and include the relevant estimates
or cases. 

\begin{theorem} (cf. [9],Theorem 1.2) Let $a,c,h \in \mathbb{N}$ with $c\ge 1$ being odd
and $1 \le a \le 2^h$. Then $S(c2^h,(c-1)2^h+a)$ is a MZC and
\[
\nu(S(c2^h,(c-1)2^h+a))= \sigma(a)-1
\]
\end{theorem}
~\\
Note. In ([9] Theorem 1.2), it is assumed that $h \ge 2$, which appears from our proof
to be an unnecessary assumption. 

\begin{theorem} (cf. [9] Theorem 1.1) Let $h,a,b,c \in \mathbb{N}$ with $0<a<2^{h+1}$,
$b2^{h+1} +a \le c2^h$ and $c\ge 1$ being odd. Then the almost minimum zero estimate
is
\[
\nu(S(c2^h,b2^{h+1}+a)) \ge \sigma(a)-1
\]
\end{theorem}

\begin{theorem} (cf.[9] Theorem 1.4) Let $a,b,c,m,h \in \mathbb{Z}^+$ with $0<a<2^{h+1}$,
$b2^{h+1}+2^h < 2^m$, and $c \ge 1$ being odd. Then if $a<2^{h+1} -1$, the almost
minimum zero estimate is $\nu(S(c2^m+b2^{h+1}+2^h,b2^{h+2}+a)) \ge \sigma(a)-1$,
which is not sharp, i.e., $\nu(S(c2^m+b2^{h+1}+2^h,b2^{h+2}+a) \ge \sigma(a)$.
If $a=2^{h+1}-1$, we have a AMZC with the same estimate, i.e., 
$\nu(S(c2^m+b2^{h+1}+2^h,b2^{h+2}+a)) = \sigma(a)-1=h$.
\end{theorem}

\begin{theorem} ([cf. [12] Theorem 6) Let $h,u,c \in \mathbb{N}$ with $u \le 2^{h}$. Then
if $u<2^h$, the shifted almost minimum zero estimate is
\[
\nu(S(c2^h+u,2^h)) \ge h-1-\nu(u)
\]
Furthermore $u$ is even and $u \le 2^{h-1}$, or $u=1$, or $u=1+2^{h-1}$
are all the cases where the estimate is sharp, i.e., where
\[
\nu(S(c2^h+u,2^h)) = h-1-\nu(u)
\]
Finally if $u=2^h$, then $\nu(S(c2^{h}+u,2^h)=0$.
\end{theorem}

\begin{theorem} (cf. [12] Theorem 7) Let $h,k,u,c \in \mathbb{N}$ with $1 \le k \le 2^h$ and
$u \le 2^{\nu(k)}$. Then if $u<k$ the shifted almost minimum zero estimate is
\[
\nu(S(c2^h+u,k) \ge \nu(k)+\sigma(k)-\nu(u)-2
\]
Furthermore we have the sharp estimate
\[
\nu(S(c2^h+u,k))=\nu(k)+\sigma(k)-\nu(u)-2
\]
if and only if $u=1$, or $1 \le u \le 2^{\nu(k)-1}$ and $u$ is even, or $u=1+2^{\nu(k)-1}$,
or $u=2^{\nu(k)}$. Finally if $u=k$, so that $\sigma(k)=1$, then $\nu(S(c2^{h}+u,k))=0$. 
\end{theorem}
~\\
Notes. Theorem (2.4) is the special case of Theorem (2.5) for $k=2^h$. The estimates
given by Lengyel in [12] are considerably weaker than ours, since $\nu(u) < \lfloor \log_2(u)
\rfloor$ unless $u=2^m$. Also he gets exact values only for the $2$-powers instead
of for all the even numbers less than or equal to $2^{\nu(k)-1}$.
 ~\\

The next asymptotic result does not depend on the estimates or cases given in the
Introduction but  depends instead on the new estimate given in the Appendix.

\begin{theorem} (cf. [12] Theorem 5)
If $\nu(k) < \nu(n)$ or if $\nu(k)=\nu(n)$ and $2^{\nu(n-k)} \in [k]$, then
\[
\lim_{h\rightarrow\infty} \nu(S(2^h n, 2^h k)) = \sigma(k)-\sigma(n) + 
 \nu\binom{n+n-k}{n}
\]
and this limit is attained if $2^{h-1+\nu(n-k)} \ge \nu\binom{n+n-k}{n}$. If $\nu(n)<\nu(k)$
or $\nu(n)=\nu(k)$ and $2^{\nu(n-k)} \in [n]$, then
\[
\lim_{h\rightarrow\infty} \nu(S(2^h n, 2^h k)) = \sigma(k-1)-\sigma(n-1) +
 \nu\binom{n-1+n-k}{n-1}
\]
Furthermore, 
if $\nu(n)<\nu(k)$, the limit is attained if $2^{h-1+\nu(n)} \ge \nu\binom{n-1+n-k}{n-1}$,
while if $\nu(n)=\nu(k)$ and $2^{\nu(n-k)} \in [n]$, the limit is attained if 
$2^{h-1+\nu(n)} > \nu\binom{n-1+n-k}{n-1}$.
\end{theorem}
~\\
Remark. Our formulas for the limit are simplier than Lengyel's, and the estimates
for when the limits are attained are exponentially better. 

\section{Basic properties of the estimates and cases, some examples and new results}

The key formula is
\begin{align}
S(n,k) = \binom{n}{k} B_{n-k}^{(-k)}
\end{align}
and since $B_{n-k}^{(-k)} = (k/n) B_{n-k}^{(-k+1)}(1)$, we get
\begin{align}
S(n,k)=\binom{n-1}{k-1}B_{n-k}^{(-k+1)}(1)
\end{align}
Hence, from our maximum pole formula, we have
\begin{align}
\nu(S(n,k)) = \nu\binom{n}{k} + \nu\left( B_{n-k}^{(-k)}\right) \ge \sigma(k)
 -\sigma(n) + \#([n] \cap [n-k])
\end{align}
and
\begin{align}
\nu(S(n,k)) \ge \sigma(k-1)-\sigma(n-1) + \#([n-1] \cap [n-k])
\end{align}

Formula (3.3) is the almost minimum zero estimate, which is sharp 
without the MZC iff the Newton
polygon of $B_{n-k}^{(-k)}(x)$ is weakly decreasing, i.e., its final segment is
horizontal. The geometry of sharpness for the shifted almost minimum zero case
is less clear, namely we may or may not have a horizontal final segment.

\begin{theorem}
The almost minimum zero and shifted almost minum zero estimates are non-negative.
\end{theorem}
\begin{proof}
The almost minimum zero estimate is
\[
\nu(S(n,k)) \ge \sigma(k)-\sigma(n) + \#([n] \cap [n-k])
\]
which is non-negative by (Lemma 2.1). The proof for the shifted estimate is identical, replacing
$(n,k)$ by $(n-1,k-1)$.
\end{proof}

Note that $\nu\left( B_{n-k}^{(-k+1)}(1)\right) = -\sigma(n-k)$ iff $\nu\left( 
B_{n-k}^{(-k+1)}\right) = -\sigma(n-k)$ iff $B_{n-k}^{(-k+1)}(x)$ is a maximum
pole case iff $S(n-1,k-1)$ is a MZC. This gives an alternative proof of the Amdeberhan
conjecture [5] which was proved in [8]. Thus $\nu(S(n+1,k+1)) = \sigma(k)-\sigma(n)$
iff $S(n,k)$ is a MZC, in which case $\nu(S(n+1,k+1))=\nu(S(n,k))$.

\begin{theorem} (Odd Stirling numbers of the second kind) The following are
equivalent:
\begin{itemize}
\item[(a)] $S(n,k)$ is odd.
\item[(b)] $\#([n] \cap [n-k]) = \sigma(n)-\sigma(k)$ and $S(n,k)$ is a MZC
or AMZC.
\item[(c)] $\#([n-1] \cap [n-k]) = \sigma(n-1)-\sigma(k-1)$ and $S(n,k)$ is a SMZC
or SAMZC.
\end{itemize}
\end{theorem}
\begin{proof}
Since $\nu(S(n,k)) \ge \sigma(k)-\sigma(n) + \#([n] \cap [n-k]) \ge 0$, we have
$\nu(S(n,k))=0$ iff the almost minimum zero estimate is sharp and zero, i.e.,
$\#([n] \cap [n-k]) = \sigma(n) -\sigma(k)$ and the estimate is sharp. 
The argument is similar for the shifted case.
\end{proof}

A different necessary and sufficient condition for $S(n,k)$ to be odd is proved in
([6], Theorem 2.1), which has no obvious relation to ours. 

The preceding theorem is particularly helpful once we have established criteria for the
different cases.

\begin{theorem} (Criteria for the four cases)
\begin{itemize}
\item[(i)] $S(n,k)$ is a MZC iff $[n-k] \cap [n] = \emptyset$.
\item[(ii)] $S(n,k)$ is a SMZC iff $[n-k]\cap [n-1] = \emptyset$.
\item[(iii)] $S(n,k)$ is a AMZC iff $[n-k]\cap [n] \ne \emptyset$ and precisely one
of the following conditions holds:
\begin{itemize}
\item[(a)] $\nu\binom{n+n-k}{n} = \#([n] \cap [n-k])$, i.e., $n+n-k$ has no unforced
carries.
\item[(b)] $\nu\binom{n+n-k-1}{n} = \#([n]\cap [n-k])-1$, i.e., $\nu(n) =\nu(n-k)$
and $n+n-k-1$ has no unforced carries.
\item[(c)] $n-k$ is odd and $\nu\binom{n+n-k-2}{n} = \#([n] \cap [n-k])-1$, i.e.,
$n-k$ is odd, the least positive exponent in $[n]$ is the same as the least positive
exponent in $[n-k]$, and $n+n-k-2$ has no unforced carries.
\end{itemize}
\item[(iv)] $S(n,k)$ is a SAMZC iff $[n-k]\cap [n-1] \ne \emptyset$ and precisely
one of the following conditions holds:
\begin{itemize}
\item[(a)] $\nu\binom{n-1+n-k}{n-1} = \#([n-1]\cap [n-k])$, i.e., $n-1+n-k$
has no unforced carries.
\item[(b)] $n-k$ is odd and 
$\nu\binom{n-1+n-k-2}{n-1} =\#([n-1]\cap [n-k])-1$, i.e., $n-k$ is odd,
the least positive exponent in $[n-1]$ is the least positive exponent in $[n-k]$, and
$n-1+n-k-2$ has no unforced carries.
\end{itemize}
\end{itemize}
\end{theorem}
~\\
Remark. Note that the shift is generally advisable only if $\nu(n) < \nu(k)$.

\begin{proof}
We omit the proof details, which follow from the material on maximum poles in the
Appendix, other than to note that in (iii), (a) comes from the partition where $u_1 =n-k$,
and (b) comes from the partition where $u_1 = n-k-1$, and (c) comes from the
partition where $u_1 = n-k-3$ and $u_3=1$. Similarly for (iv). Conditions $(b)$ and $(c)$
cannot both hold by Lemma 2.2, since if $2^0 \in [n] \cap [n-k]$ and $2^1 \in [n]-
[n-k-2]$ or $2^1 \in [n-k-2]-[n]$, then $n+n-k-2$ has an unforced carry in place $2^1$.
Lemma 2.2 also eliminates the a priori possibility that the partition where $u_1=
n-k-4$ and $u_3=1$, with $d=n-k-3$ being odd, gives the maximum pole. Thus the
partitions noted in (iii) or (iv) are the only ones that can give the maximum pole
of $B_{n-k}^{(-k)}(x)$ or $B_{n-k}^{(-k+1)}(x)$.
\end{proof}

\begin{corollary} (Hong-Amdeberhan [5,8]) 
$\nu(S(n,k)) = \sigma(k-1)-\sigma(n-1)$ iff $S(n,k)$ is a SMZC iff $S(n-1,k-1)$ is a
MZC iff $\binom{n-1+n-k}{n-1}$ is odd.
\end{corollary}

\begin{corollary} (Central Stirling numbers)
For the central Stirling number $S(2k,k)$ the almost minimum zero estimate is
\begin{itemize}
\item[(a)] $\nu(S(2k,k)) \ge \#(\text{pairs of consecutive ones in~} k)$, and 
\item[(b)] $\nu(S(2k,k)) = 1$ iff $S(2k,k)$ is a AMZC iff $k = 3+8k'$ with $k'$ Fibbinary.
\end{itemize}
\end{corollary}
\begin{proof}
For (a), if $n=2k$, then $\sigma(n) = \sigma(k)$ and $\#(\text{pairs of consecutive 
ones in~k})\linebreak
= \#([k] \cap [2k]) = \#([n] \cap [n-k])$ so the inequality is just the almost zero minimum 
zero estimate. For (b), we have $\nu(S(2k,k)) =1$ iff $\#([n] \cap [n-k]) = 1$ iff $S(n,k)$ is a AMZC.
If $k=3+8k'$ where $k'$ is Fibbinary then $n-k=k$ is odd, so the least positive exponent in
$k$ is $1$, which is the least positive exponent in $n$. Also $n+n-k-2 = 6+16k' +1+
8k' = 2+4+8(2k' +k')+1$ has no carries, so (iii) part (c) applies. Finally, if there is a 
different pair of consecutive ones in $k$, it is easy to see that none of the conditions in
(iii) apply, so $S(n,k)$ is not a AMZC.
\end{proof}

The parts of the next theorem can be found in the literature, e.g., in ([9] and [12]).
It is included here as an excellent example of our estimates and cases.

\begin{theorem}
Let $c \ge 3$ be odd, $n=c2^h$, and $1 \le k \le 2^{h+1}$. Then
\begin{itemize}
\item[(i)] If $k \le 2^h$ then $S(n,k)$ is a AMZC and $\nu(S(n,k)) = \sigma(k)-1$.
(Lengyel's extension of De Wannemacker's theorem.)
\item[(ii)] If $2^h < k < 2^{h+1}$ and $k=2^h+a$, then the almost minimum zero 
estimate is $\nu(S(n,k)) = \nu(S(n,2^h+a)) \ge \sigma(a) = \sigma(k)-1$.
\item[(iii)] If $k=2^h+a$ with $0<a<2^h -1$, so $k <2^{h+1}-1$, then $S(n,k)$
is not a AMZC, so $\nu(S(n,k)) \ge \sigma(a)+1 = \sigma(k)$, while if $a=2^h-1$,
so that $k = 2^{h+1}-1$, then $S(n,k)$ is a AMZC , and $\nu(S(n,k)) = \sigma(a)=h$.
\item[(iv)] If $a=2^h$, i.e., $k=2^{h+1}$, then $S(n,k)$ is a AMZC and SAMZC,
and $\nu(S(n,k))=0$.
\end{itemize}
\end{theorem}

\begin{proof}
For (i), we have $n-k=(c-1)2^h+2^h-k$, so $\#([n-k]\cap [n])=\sigma(c)-1$. For the
sum $n-k+n$, the carries are the same as for $c-1+c$, which are all unforced. Also
$\nu(n-k) \ne \nu(n)=h$, and the smallest positive exponent in $n-k$ is not equal to
the smallest positive exponent in $n$. Thus $S(n,k)$ is a AMZC with $\nu(S(n,k)) = 
\sigma(k)-\sigma(n)+\sigma(c)-1=\sigma(k)-\sigma(c)+\sigma(c)-1=\sigma(k)-1$.

For (ii), if $k=2^h+a$ with $0< a<2^h$, then $\sigma(k)-\sigma(n)=1+\sigma(a)
-\sigma(c)$ and if $\alpha =\nu(c-1)$ and $T=c-1$, then $n-k=(T-2^\alpha)2^h
+(2^{\alpha+h}-2^h) +2^h -k = (T-2^\alpha)2^h + 2^{\alpha+h-1} + \cdots+2^h +2^h-a$,
so $[n-k] \cap [n] = [T] - \{2^{\alpha +h}\} \cup \{2^h\}$ and $\#([n-k] \cap [n])=
\sigma(c)-1$. Thus the minimum zero estimate for $S(n,k)$ is $1+\sigma(a)-\sigma(c)
+\sigma(c)-1 = \sigma(a)=\sigma(k)-1$.

For (iii), if $0<a<2^h-1$ then $n-k+n$ has an unforced carry for exponent $\alpha$,
and the other partitions in Theorem 3.3(iii) are also not valid, so $S(n,k)$ is not a AMZC,
and $\nu(S(n,k)) \ge \sigma(a)+1 = \sigma(k)$. If $a=2^h-1$, so $k=2^{h+1}-1$, it is 
easy to verify that the first two partitions in Theorem 3.3(iii) still fail to meet the
conditions, but the third
partition, when $u_1=n-k-3$ and $u_3=1$ now works, so $S(n,k)$ is a AMZC, and
$\nu(S(n,k)) = \sigma(a)=\sigma(2^h-1)=h=\sigma(k)-1$. 

For (iv), if $a=2^h$, i.e., $k=2^{h+1}$, we now have $n-k=(T-2^\alpha)2^h+
2^{\alpha+h-1} + \cdots+ 2^h$, so the partitions of type (a) and (c) in Theorem 3.3(iii)
now fail, but the partition of type (b) where $u_1=n-k-1$ works, so $S(n,k)$ is a AMZC
and $\nu(S(n,k)) = \sigma(k)-\sigma(n)+\#([n] \cap [n-k]) = 1-\sigma(c)+\sigma(c)-1=0$.
Finally, the shifted minimum zero estimate is $\nu(S(n,k)) \ge \sigma(k-1)-\sigma(n-1)
+\#([n-1] \cap [n-k]) = \sigma(k)-1+\nu(k)-(\sigma(n)-1+\nu(n)) + \sigma(c)-2=0$,
and now the partition of type (a) works and the partition of type (c) doesn't (since
$n-k$ is even), so $S(n,k)$ is a SAMZC.
\end{proof}

The following theorem can be easily proved using the criteria for the cases, so will not
give the proof. It does show that Lengyel's extension of De Wannemacker's Theorem
follows formally from DeWannemacker's Theorem.

\begin{theorem} (Invariance)
Suppose $\Delta>0$ and all $2$-powers in $\Delta$ are greater than all $2$-powers in $n$. 
Then
\begin{itemize}
\item[(a)] for all four estimates, the estimate for $\nu(S(n,k))$ is the same as the 
estimate for $\nu(S(n+\Delta,k))$ and also the same for $\nu(S(n+\Delta,k+\Delta))$.
\item[(b)] $S(n+\Delta,k)$ is a AMZC iff $S(n,k)$ is a AMZC or MZC, and if any
of the cases hold, then $\nu(S(n+\Delta,k))=\nu(S(n,k))$.
\item[(c)] The same results hold if we replace cases by their shifts.
\item[(d)] If $\nu(\Delta) > \lfloor \log_2(n) \rfloor +1$, then $S(n+\Delta,k+\Delta)$
is a AMZC if $S(n,k)$ is a MZC or AMZC. Similarly for the the shifts. For all of these cases,
we have $\nu(S(n+\Delta,k+\Delta))=\nu(S(n,k))$.
\end{itemize}
\end{theorem}
~\\
Remark. The assumption in (d) gives a ``gap" in the $2$-powers between $n$ and $n+\Delta$.
This is necessary to preserve the no unforced carries conditiion as we pass from $(n,k)$
to $(n+\Delta,k+\Delta)$.

For reference purposes, we include the basic material about Stirling numbers of the
first kind $s(n,k)$:
\begin{align}
s(n,k)=\binom{n-1}{k-1}B_{n-k}^{(n)}
\end{align}
Thus by the recursive formula (6.3), we get
\begin{align}
s(n,k)=\binom{n}{k}B_{n-k}^{(n+1)}(1)
\end{align}
From the Appendix, the maximum pole of $B_{n-k}^{(n)}(x)$ is 
$\#([n-k]\cap [k-1])$, so we get the following four estimates:\\
Minimum zero estimate:
\begin{align}
\nu(s(n,k)) \ge \sigma(k-1)-\sigma(n-1)
\end{align}
Shifted minimum zero estimate:
\begin{align}
\nu(s(n,k)) \ge \sigma(k)-\sigma(n)
\end{align}
Almost minimum zero estimate:
\begin{align}
\nu(s(n,k)) \ge \sigma(k-1)-\sigma(n-1)+\#([n-k]-[k-1])
\end{align}
Shifted almost minimum zero estimate:
\begin{align}
\nu(s(n,k)) \ge \sigma(k)-\sigma(n) +\#([n-k]-[k])
\end{align}

\section{Proofs of  theorems 2.1-2.5}

\textit{Proof of Theorem 2.1}.
Let $n=c2^h$ and $k=(c-1)2^h+a$, with $1 \le a \le 2^h$. $n-k=2^h-a$, so $[n]\cap [n-k]
=\emptyset$, so $S(n,k)$ is a MZC and $\nu(S(n,k)) = \sigma(k)-\sigma(n)=\sigma(c-1)
+\sigma(a)-\sigma(n)=\sigma(a)-1$. (If $a=2^h$, the theorem is trivial.)

\vspace{-8mm}

\begin{flushright}
$\Box$
\end{flushright}

\pagebreak

~\\
\textit{Proof of Theorem 2.2}.
Let $n=c2^h$ and $k=b2^{h+1} +a$. Then $\sigma(k)-\sigma(n)=\sigma(a)+\sigma(b)
-\sigma(c)$, so by the almost minimum zero estimate, it will suffice to show that
$\#([n] \cap [n-k]) \ge \sigma(c)-\sigma(b)-1$. 

If $a \le 2^h$, then $n-k=(c-2b-1)2^h+2^h-a$,
so $[n] \cap [n-k]=[c2^h] \cap \left([(c-2b-1)2^h] \cup [2^h-a]\right) = [c2^h] \cap
[(c-2b-1)2^h] = [c] \cap [c-2b-1]$, so by Lemma 2.1, we have 
$\#([n] \cap [n-k]) \ge \sigma(c)-\sigma(2b+1)
=\sigma(c)-\sigma(b)-1$, which completes the proof in this case.

On the other hand if $a>2^h$, then $n-k=(c-2b-2)2^h +2^{h+1}-a$ and $0\le 2^{h+1}-a
<2^h$, so $[n-k] \cap [n] = [(c-2b-2)2^h] \cap [c2^h] =[c] \cap [c-2b-2]  =
([1] \cup [c-1]) \cap ([1] \cup [c-1-(2b+1)])$ so $\#([c] \cap [c-2b-2]) \ge 1+\sigma(c-1)
-\sigma(2b+1) = \sigma(c)-\sigma(b)-1$.

\vspace{-4mm}

\begin{flushright}
$\Box$
\end{flushright}

\vspace{-6mm}

~\\
\textit{Proof of Theorem 2.3}.
The proof of this theorem is similar to our Theorem (3.4). Let $n=c2^m+b2^{h+1}+2^h$ 
and $k=b2^{h+2}+a$, where $0<a<2^{h+1}$. First assume $0<a \le 2^h$. Then
$n-k=(c-1)2^m +2^m-b2^{h+1}+2^h-a$, so $[n-k] \cap [n] = ([(c-1)2^m] \cap [c2^m])
\cup ([b2^{h+1}] \cap [2^m-b2^{h+1}])$, so $\#([n-k] \cap [n]) = \sigma(c)-1+1=\sigma(c)$.
Thus the almost minimum zero estimate is $\nu(S(n,k)) \ge \sigma(b)+\sigma(a)-(\sigma(c)
+\sigma(b)+1) +\sigma(c)=\sigma(a)-1$. But $n+n-k$ has an unforced carry
for exponent $m$ and $\nu(n-k)
\ne \nu(n)$ and the first positive exponent in $n-k$ is not equal to the first positive exponent
in $n$, so $S(n,k)$ is not a AMZC, by the criteria.

Next assume $2^h<a<2^{h+1}$. Then $n-k =(c-1)2^m +2^m-(2b+1)2^h+2^{h+1}-a$,
so again $\#([n-k] \cap [n]) = \sigma(c)$. If $a<2^{h+1}-1$, then once more the three
partitions don't satisfy the AMZC criterion. Finally if $a=2^{h+1}-1$, so that $n-k=
(c-1)2^m+2^m-(2b+1)2^h+1$, then the partitions when $u_1=n-k$ and when $u_1=n-k-1$
fail the the criteria, but the partition when $u_1=n-k-3$ and $u_3=1$ does meet the
criteria. Hence $S(n,k)$ is a AMZC when $a=2^{h+1}-1$, and $\nu(S(n,k))=\sigma(a)-1=h$.

\vspace{-6mm}

\begin{flushright}
$\Box$
\end{flushright}

\vspace{-2mm}

Since Theorem 2.4 is a special case of the next one, we will not prove it.

~\\
\textit{Proof of Theorem 2.5}.
Let $n=c2^h+u$ and $1 \le k \le 2^h$ with $0 < u \le 2^{\nu(k)}$. Then $n-k=(c-1)2^h+u
+2^h-k$. 
Without loss of generality, we can assume $c$ is odd. Then $n-1=c2^h+u-1$ so $\sigma(k-1)
-\sigma(n-1)=\sigma(k-1)-(\sigma(c)+\sigma(u-1))=
\sigma(k)-1+\nu(k)-\sigma(c)-\sigma(u)+1-\nu(u)=\sigma(k)+
\nu(k)-\nu(u)-\sigma(c)-\sigma(u)$. Also $[n-1] \cap [n-k] = [(c-1)2^h]\cup ([u] \cap [u-1])$,
so $\#([n-1] \cap [n-k] = \sigma(c)-1+\sigma(u)-1$. Therefore, the shifted almost minimum zero
estimate if $u\ne k$ is
\begin{align}
\nu(S(n,k)) &\ge \sigma(k)+\nu(k)-\nu(u)-\sigma(c)-\sigma(u)+\sigma(c)+\sigma(u)-2 \notag\\
&= \nu(k)+\sigma(k)-\nu(u)-2 \notag
\end{align}
It remains to show that $S(n,k)$ is a SAMZC (sharp estimate) 
iff $u=1$, or $u$ is a positive even integer which is less 
than or equal to $2^{\nu(k)-1}$, or $u=1+2^{\nu(k)-1}$, or $u=2^{\nu(k)}<k$. 
But $n-k+n-1=((c-1)2^h+(2^h-k)+u)+(c2^h+u-1)$, which has
no unforced carry as long as $u \le 2^{\nu(k)-1}$. Thus we have a SAMZC 
(sharp estimate) iff the partition where
$u_1=n-k-3$ and $u_3=1$ fails the criterion. If $u$ is even then $n-k$ is even, so this
partition fails while if $u$ is odd and $u \ne 1$, then the criterion is met, so again we don't have a
sharp estimate unless $u=1+2^{\nu(k)-1}$. For all other $u$, the criterion for a sharp estimate 
fails. This proof illustrates the fact that precisely one of the 
partitions must satisfy the criterion for a sharp estimate.

It is easy to see that if $u=k$ so that $\sigma(k)=1$, then $\nu(S(c2^h+u,k))=0$.

\vspace{-6mm}

\begin{flushright}
$\Box$
\end{flushright}

\section{Proof of the asymptotic theorem 2.6}
~\\
\textit{Proof of Theorem 2.6}. We use the notations of the Appendix.
(i) First consider the case where $\nu(k) < \nu(n)$, so $\nu(n-k)=\nu(k)$ and $2^{\nu(k)} 
\notin [n]$. Let $w(u) \le n-k$. Then if $d=n-k-2^{\nu(k)}$, the number of carries for
$d+n$ is the same as the number of carries for $n-k+n$, so if $n-k \ge d \ge n-k-2^{\nu(k)}$
then $\nu\binom{d+n}{n} \ge \nu\binom{n-k+n}{n}$. Since $n-k-\nu(u) \ge 0$, with
equality iff $u_1=n-k$, it follows that $\nu(t_u) \ge \nu\binom{n-k+n}{n}$, with equality
iff $u_1=n-k$. Thus the single partition $u_1=n-k$ has the least $2$-adic value among all
these terms in this case.

If on the other hand $d < n-k-2^{\nu(k)}$ then $n-k-d > 2^{\nu(k)}$. If $h$ is such
that $2^{\nu(k)+h-1} \ge \nu\binom{n-k+n}{n}$, replace $(n,k)$ by $(2^hn,2^hk)$.
Then by Corollary 6.1, we have $n-k-\nu(u) > \nu\binom{n-k+n}{n}$, so again
$\nu(t_u) > \nu\binom{n-k+n}{n}$.
Therefore the single partition $u_1=n-k$ gives the least value if
$2^{\nu(k)+h-1} \ge \nu\binom{n-k+n}{n}$, and $\nu(S(2^h n,2^h k)) =\sigma(k)
-\sigma(n)+\nu\binom{n+n-k}{n}$.

(ii) The case where $\nu(n)=\nu(k)$ and $2^{\nu(n-k)} \in [k]$ is similar, namely in this
case $2^{\nu(n-k)} \notin [n]$, so we get the same value for $\nu(S(2^h n,2^h k))$ if
$2^{\nu(n-k)+h-1} \ge \nu\binom{n-k+n}{n}$.

(iii) If $\nu(k)>\nu(n)$, then $\nu(n-k)=\nu(n)$, and we consider $\nu(S(n,k))=
\nu\binom{n-1}{k-1} +\nu\left(B_{n-k}^{(-k+1)}(1)\right)$. Since $2^{\nu(n-k)}=
\text{min}([n])$, we now have $2^{\nu(n-k)} \notin [n-1]$, so essentially the same
argument shows that the term when $u_1=n-k$ is the single dominant term if
$2^{\nu(n)+h-1} \ge \nu\binom{n-1+n-k}{n-1}$ and $\nu(S(2^h n,2^h k)) = \sigma(k-1)
-\sigma(n-1)+\nu\binom{n-1+n-k}{n-1}$.

(iv) The final case, when $\nu(n)=\nu(k)$ and $2^{\nu(n-k)} \in [n]$ is slightly more
delicate. In this case $\nu(n) < \nu(n-k)$, and if $d> n-k -2^{\nu(n)} =n-k-2^{\nu(n-k)}+
2^{\nu(n-k)}-2^{\nu(n)}$ then $d=n-k-2^{\nu(n-k)}
+2^{\nu(n)} +\delta$, where $0< \delta< 2^{\nu(n)}$. Then $[\delta] \cap [n-1]
\ne \emptyset$, so if
we consider $d+n-1$, we get an unforced carry in power $2^{\nu(n)}$, which in turn
leads to an unforced carry in power $2^{\nu(n-k)}$. Thus $\nu\binom{d+n-1}{n-1}
\ge \nu\binom{n-k+n-1}{n-1}$. Finally, if $2^{\nu(n)+h-1} > \nu\binom{n-1+n-k}{n-1}$
then the terms when $d \le n-k-2^n$ have bigger value, so again $\nu(S(2^h n,2^h k))=
\sigma(k-1)-\sigma(n-1)+\nu\binom{n-1+n-k}{n-1}$. The partition when $u_1=n-k$
is again dominant.

\vspace{-4mm}

\begin{flushright}
$\Box$
\end{flushright}

\vspace{-2mm}

\begin{corollary} (Central Stirling numbers)
$\lim_{h \rightarrow\infty} \nu(S(2^{h+1}k,2^hk)) = \nu\binom{3k}{k}$ and
$\nu(S(2^{h+1}k,2^hk)) = \nu\binom{3k}{k}$ if $2^{h-1+\nu(k)} \ge \nu\binom{3k}{k}$.
\end{corollary}

\begin{proof}
This follows immediately from the first case of the preceding theorem.
\end{proof}
~\\
Remark. Since $\binom{n+n-k}{n} = \frac{n+n-k}{n}\binom{n-1+n-k}{n-1}$, if
$\nu(k) < \nu(n)$ then $\nu\left(\frac{n+n-k}{n}\right) = \nu\left(\frac{k}{n}\right)
=\nu(k)-\nu(n)$ so $\sigma(k)-\sigma(n) +\nu\binom{n+n-k}{n} = \sigma(k-1)-
\sigma(n-1)+\nu\binom{n-1+n-k}{n-1}$. If $\nu(k)=\nu(n)$ then 
$\nu\left(\frac{n+n-k}{n}\right)=0$ and $\sigma(k-1)-\sigma(n-1)=\sigma(k)-\sigma(n)$, 
so again we get the same value for $\nu(S(2^h n, 2^h k))$. Thus
\[
\lim_{h\rightarrow\infty} \nu(S(2^h n,2^h k)) = \sigma(k-1)-\sigma(n-1)+
 \nu\binom{n-1+n-k}{n-1}
\]
in all cases, which can easily be shown to agree with Lengyel's limit.

\section{Appendix --- Higher order Bernoulli numbers and polynomials}

The higher order Bernoulli polynomials $B_n^{(l)}(x)$ are defined by
\begin{align}
\left(\frac{t}{e^t-1}\right)^l e^{tx} = \sum_{n=0}^\infty B_n^{(l)}(x)t^n/n!
\end{align}

In this paper we assume the order $l \in \mathbb{Z}$. If $x=0$, we get the higher
order Bernoulli numbers $B_n^{(l)}$, and we get the Appell property $B_n^{(l)}(x) =
\sum_{i=0}^n \binom{n}{i} B_{n-i}^{(l)} x^i$. This polynomial is rational, monic
and in $\mathbb{Q}[x]$.

These polynomials satisfy two recursions:
\begin{align}
\left(B_n^{(l)}(x)\right)' =nB_{n-1}^{(l)}(x) ~ \text{and} ~ \Delta(B_n^{(l)}(x))
 = B_n^{(l)}(x+1)-B_n^{(l)}(x) =nB_{n-1}^{(l-1)}(x)
\end{align} 
These recursions yield the recursive formula
\begin{align}
B_n^{(l)} = \frac{l}{l-n} B_n^{(l+1)}(1)
\end{align}

If $u=(u_1,u_2, \ldots)$ is a sequence of natural numbers eventually zero, we regard $u$
as a partition of the number $w=w(u)=\sum i u_i$, where $u_i$ is the multiplicity of
$i$ in the partition and $d=d(u)=\sum u_i$ is the number of summands.

There is an explicit representation of $B_n^{(l)}$ in terms of the partitions, namely
(cf. [1,2])
\begin{align}
B_n^{(l)} = (-1)^n n! \sum_{w \le n} t_u(l-n-1)
\end{align}
and also
\begin{align}
B_n^{(l)}(1) =(-1)^n n! \sum_{w=n} t_u(l-n)
\end{align}
where $t_u=t_u(s)=\binom{s}{d}\binom{d}{u}/\Lambda^u$, where $\binom{d}{u} = 
\binom{d}{u_1 u_2 \ldots}$
is a multinomial coefficient, $\Lambda^u = 2^{u_1}3^{u_2} \cdots$, and
$\nu(u)=\nu(\Lambda^u)=\sum u_i\nu(i+1)$.

There is a companion sequence $\tau_u = \tau_u(s) = (n)_w t_u$, where
$s=l-n-1$, which is important
for the study of the $B_n^{(l)}(x)$. In particular, the maximum pole of $B_n^{(l)}(x)$
is the maximum pole of $\{\tau_u(l-n-1) : w \le n\}$. In [1] we showed that for $p=2$
the maximum pole of $B_n^{(l)}(x)$ is  $\#\{2^i \in [n] : 2 \nmid \binom{l-n-1}{2^i}\}$.
We can use the same reduction method for $p=2$ as in the proof of ([1], Lemma 3.1)
to show that if $\tau_u$ has the maximum pole, then $u_i=0$ for all $i>1$, with the
possible exception $u_3=1$, i.e., $u$ is concentrated in places $1$ and $3$, with $u_3\le 1$:
If $i \ne 1,3$ and $u_i\ne 0$ or if $i=3$ and $u_i \ge 2$, delete $u_i$ and increase
$u_1$ by $u_i$. (We call this a transfer from place $i$ to place $1$.) This preserves $d$
and decreases $w$. It is easy to see that this also decreases $\nu(\tau_u)$, so is 
impossible if $\tau_u$ has the maximum pole. 

Since $n! t_u =(n-w)! \tau_u$, we see that $n! t_u$ has the maximum pole of $B_n^{(l)}(x)$
iff $\tau_u$ has the maximum pole and $w=n-1$ or $w=n$.

For our application to Stirling numbers of the second kind, we replace $n$ by $n-k$ and $l$
by $-k$. It follows that the maximum pole is $\#([n-k]-[n])$, and by our analysis (cf. [1]), the
first pole has order one, and occurs in codegree of the smallest element of $[n-k]-[n]$,
etc. That is how we get the Newton polygon of the higher order Bernoulli polynomial,
which is particularly simple, the poles coming in increasing order without gaps (cf. [3]).
Newton polygons are used in [10] in a different way.

Furthermore from our analysis of the possible maximum pole terms, we can show that 
$B_{n-k}^{(-k)}$ has the maximum pole iff precisely one of the following terms gives
the maximum pole:
\begin{align}
\text{(i)~} &u_1=n-k, \text{~so~} w=n-k=d \text{~and~}\notag\\
 &\qquad t_u=(-1)^{n-k} \binom{n+n-k}{n}/2^{n-k}\\
\text{(ii)~} &u_1 =n-k-1, \text{~so~} w=n-k-1=d \text{~and~}\notag\\
&\qquad t_u=(-1)^{n-k-1} \binom{n+n-k-1}{n}/2^{n-k-1} \\
\text{or (iii)~} &u_1=n-k-3 \text{~and~} u_3=1, \text{~so~} w=n-k \text{~and~} d=n-k-2
\notag\\
&\qquad \text{~~~and~} n-k \text{~is odd and greater than 1 and}\notag\\
&\qquad t_u=(-1)^{n-k} \binom{n+n-k-2}{n}(n-k-2)/2^{n-k-1}
\end{align}
~\\
Remark. These three partitions are the ones that determine the mod $4$ congruence for
$2^{n-k} B_{n-k}^{(-k)}/(n-k)!$. The a priori possible term with $u_1=n-k-4$ and
$u_3=1$ is eliminated in the proof of Theorem 3.3.

Finally, we give a new estimate that is very useful for our asymptotic analysis.

\begin{theorem}
Let $w \le n$. Then $n- \nu(u) \ge n-w +(w-d)/2$ and $n-\nu(u) =(w-d)/2$ iff $n=w$ and 
$u$ is concentrated in places $1$ and $3$.  
\end{theorem}
\begin{proof}
Since $n-\nu(u)= n-w+w-\nu(u)$,
it will suffice to prove that $w-\nu(u) \ge (w-d)/2$, with equality iff $u$ is concentrated
in places $1$ and $3$, i.e., we can assume $w=n$. But $w-\nu(u)-(w-d)/2 = 
\sum u_i (i-\nu(i+1) -(i-1)/2) = \frac{1}{2} \sum u_i(i+1-2\nu(i+1)) = \frac{1}{2}
\sum u_i(j -2\nu(j))$, where $j=i+1$. But if $j>0$, it is easy to see that $j \ge 2\nu(j)$ with
equality iff $j=2$ or $j=4$.
\end{proof}

\begin{corollary}
If $w \le n$ then $n-\nu(u) \ge (n-d)/2$.
\end{corollary}
~\\
Note: For our applications we will often only have an estimate for $n-d$, so this is
how typically we will use Theorem 6.1.
~\\
\centerline{ACKNOWLEDGEMENTS}
~\\
The author would like to thank E.  A. Herman for his invaluable help in preparing
this paper, and T. Lengyl for his generous advice and encouragement throughout its
development.
~\\
\centerline{REFERENCES}
~\\
1. A. Adelberg, On the degrees of irreducible factors of higher order Bernoulli 
polynomials, \textit{Acta Arith.} \textbf{62} (1992), 329-342.\\
2. A. Adelberg, Congruences of $p$-adic integer order Bernoulli numbers,
\textit{J. Number Theory} \textbf{59} No. 2 (1996), 374-388.\\
3. A. Adelberg, Higher order Bernoulli polynomials and Newton polygons,
G. E. Bergum et al (eds.), \textit{Applications of Fibonacci Numbers} \textbf{7}
(1998), 1-8.\\
4. A. Adelberg, The $p$-adic analysis of Stirling numbers via higher order
Bernoulli numbers, \textit{Int. J. Number Theory} \textbf{14} (2018), No. 10,
2767-2779.\\
5. T. Amdeberhan, D. Manna and V. Moll, The $2$-adic valuation of Stirling
numbers, \textit{Experimental Math.} \textbf{17} (2008), 69-82.\\
6. O-Y. Chan and D. Manna, Divisibility properties of Stirling numbers of the
second kind, Proceedings of the Conference on Experimental Math., T. 
Amdeberhan, L. A. Medina, and V. Moll eds., \textit{Experimental Math.}
(2009).\\
7. S. De Wannemacker, On $2$-adic orders of Stirling numbers of the second
kind, \textit{Integers Electronic Journal of Combinatorial Number Theory},
\textbf{5} (1) (2005), A21, 7 pp. (electronic).\\
8. S. Hong J. Zhao, and W. Zhao, The $2$-adic valuations of Stirling numbers of
the second kind, \textit{Int. J. Number Theory} \textbf{8} (2012), 1057-1066.\\
9. S. Hong, J. Zhao, and W. Zhao, Divisibility by $2$ of Stirling numbers of the
second kind and their differences, \textit{J. Number Theory} \textbf{140}
(2014), 324-348.\\
10. T. Komatsu and P. T. Young, Exact $p$-adic valuations of Stirling numbers
of the first kind, \textit{J. Number Theory} \textbf{177} (2017), 20-27.\\ 
11. T. Lengyel, On the divisibility by $2$ of the Stirling numbers of the second
kind, \textit{Fibonacci Quart.} \textbf{32} (3) (1994), 194-201.\\
12. T. Lengyel, On the $2$-adic order of Stirling numbers of the second kind and
their differences, \textit{DMTCS Proc. AK} (2009), 561-572.\\
13. T. Lengyel, Alternative proofs on the $2$-adic order of Stirling numbers
of the second kind, \textit{Integers} \textbf{10} (2010), A38, 453-468.\\

\end{document}